\documentclass[12pt,reqno]{amsart}

\usepackage{color}
\usepackage{amsmath, amsthm, amssymb}
\usepackage{amsfonts}

\usepackage{hyperref}

\usepackage[ansinew]{inputenc}
\usepackage[dvips]{epsfig}
\usepackage{graphicx}
\usepackage[english]{babel}

\theoremstyle{plain}
\newtheorem{thm}{Theorem}[section]
\newtheorem{cor}[thm]{Corollary}
\newtheorem{lem}[thm]{Lemma}
\newtheorem{claim}[thm]{Claim}
\newtheorem{prop}[thm]{Proposition}

\theoremstyle{definition}
\newtheorem{defn}[thm]{Definition}

\theoremstyle{remark}
\newtheorem{rem}[thm]{Remark}

\numberwithin{equation}{section}

\newcommand{\average}{{\mathchoice {\kern1ex\vcenter{\hrule height.4pt
width 6pt depth0pt} \kern-9.7pt} {\kern1ex\vcenter{\hrule
height.4pt width 4.3pt depth0pt} \kern-7pt} {} {} }}

\newcommand {\R}{\mathbb{R}}
\newcommand {\cE}{\mathcal{E}}
\newcommand {\cK}{\mathcal{K}}
\newcommand {\cC}{\mathcal{C}}

\newcommand {\bs}[1]{\boldsymbol{#1}}
\newcommand {\unl}[1]{\underline{#1}}
\newcommand {\littint}{\textstyle\int\displaystyle}

\begin{document}

\title[Sums of fractional Laplacians and phase transitions]
{An extension problem for sums of fractional Laplacians and 1-D symmetry of phase transitions}

\author{Xavier Cabr\'e}
\address{ICREA and Universitat Polit\`ecnica de Catalunya\\ Departament de Matem\`atica Aplicada I
\\
Diagonal 647, 08028 Barcelona, Spain}
\email{xavier.cabre@upc.edu}

\author{Joaquim Serra}
\address{Universitat Polit\`ecnica de Catalunya\\ Departament de Matem\`atica Aplicada I
\\
Diagonal 647, 08028 Barcelona, Spain}
\email{joaquim.serra@upc.edu}

\thanks{The authors were supported by MINECO grants MTM2011-27739-C04-01 and MTM2014-52402-C3-1-P}

\begin{abstract}
We study nonlinear elliptic equations for operators
corresponding to non-stable L\'evy diffusions.
We include a sum of fractional Laplacians of different orders.
Such operators are infinitesimal generators of non-stable (i.e., non self-similar) L\'evy processes.
We establish the regularity of solutions, as well as sharp energy estimates.
As a consequence, we prove a 1-D symmetry result for monotone solutions to Allen-Cahn type
equations with a non-stable L\'evy diffusion. These operators may still be realized as local operators
using a system of PDEs ---in the spirit of the extension problem of Caffarelli and Silvestre.
\end{abstract}

\date{}
\maketitle
\vspace {-.2in}
\begin{center}
{\em To Juan Luis V\'azquez, with friendship and admiration.}
\end{center}

\maketitle
\parskip = 0pt
\parindent = 12pt

\markboth{\hfill SUMS OF FRACTIONAL LAPLACIANS AND PHASE TRANSITIONS \hfill}{\hfill
SUMS OF FRACTIONAL LAPLACIANS AND PHASE TRANSITIONS\hfill}

\section{Introduction}

In this paper we study {\em layer solutions} of phase transition problems with a nonlocal diffusion.
The main novelty is that the diffusion operator that we consider does not have self-similarity properties.
For instance, we consider the nonlocal Allen-Cahn type equation
\begin{equation}\label{problema0}
\sum_{i=1}^K \mu_i (-\Delta)^{s_i} u + W'(u) = 0 \quad\text{in } \R^n\,,
\end{equation}
where $\mu_i>0$, $\sum \mu_i =1$, $0<s_1<\dots <s_K \le 1$, and $W$ is a double-well potential with wells of the same height located at $\pm 1$.
By definition, a {\em layer solution} is a solution which is monotone in the $x_n$ direction with limits $\pm 1$ as $x_n\to \pm \infty$. That is,
\begin{equation} \label{layer}
 u_{x_n}\ge 0 \quad\mbox{in }\R^n\quad \mbox{ and }\quad  \lim_{x_n\to \pm \infty} u(x', x_n) = \pm 1 \quad \mbox{for all }x'\in \R^{n-1}.
\end{equation}

Having always \eqref{problema0} in mind, we actually consider the more general equation
\begin{equation}\label{problema1}
L u +  W'(u) = 0 \quad\text{in } \R^n\,,
\end{equation}
where, for some $s_*\in (0,1)$, we have
\begin{equation}\label{int_lap}
L u =\int_{[s_*,1]}  (-\Delta)^{s} u \,\,\, d\mu(s).
\end{equation}
We assume that $\mu$ is a probability measure supported in $[s_*,1]$, i.e.,
\[ \mu \ge 0 \quad \mbox{and}\quad \mu\bigl([s_*,1]\bigr) = \mu(\R)=1.\]
The operator $L$ is the infinitesimal generator of a L\'evy process $Y(t)$ which is isotropic but not stable. 
It has different behaviors at large and small time scales. Heuristically, for a very small time step $h$, 
the distribution of  $Y(t+h)-Y(t)$ is, with probability $\mu\bigl([s,s+ds)\bigr)$,
that of a $2s$-stable L\'evy process.
This gives a probabilistic interpretation of $\mu$.

Recall that the fractional Laplacian is defined by
\begin{equation}\label{constant-cns}
(-\Delta)^s u(x) = c_n(s)\,\mbox{PV}\negmedspace\int_{\R^n} \frac{u(x)-u(y)}{|x-y|^{n+2s}}\,dy\,, 
\end{equation}
where
\begin{equation}\label{eq_c_n(s)}
c_n(s) = \pi^{-\frac{n}{2}} 2^{2s} \frac{\Gamma\bigl(\frac{n+2s}{2}\bigr)}{\Gamma(2-s)} s(1-s)\,.
\end{equation}
Equivalently, $(-\Delta)^s$ is the operator whose Fourier symbol is $|\xi|^{2s}.$

We may assume that
\begin{equation}\label{def-sstar}
s_* = \max\{s\,:\, \text{support}\,\mu \subset[s,1]\}\,.
\end{equation}
In the case of problem \eqref{problema0}, we have $s_*=s_1$, which is the relevant exponent in a  blow-down of the equation.

The double-well potential $W$ is assumed to satisfy
\begin{equation}\label{assumptionsW}
W\in C^3(\R)\,,\quad W(\pm1)=0\quad\mbox{and}\quad W(t)>0 \,\mbox{ for }\,t\neq\pm 1\,.
\end{equation}

Similarly as for scale invariant diffusions in \cite{SV1,SV2}, the appropriate energy functional for our problem is
\begin{equation}\label{energyfunct}
\cE(u,\Omega) =  \cK(u,\Omega)  + \int_{\Omega} W(u)\,dx, \quad \mbox{with }\cK(u,\Omega) = \littint \cK^s(u,\Omega) \,d\mu(s)\,,
\end{equation}
where, for $0<s<1$,
\begin{equation}\label{energyfunct1}
\cK^s(u,\Omega) = \frac{c_n(s)}{4} \iint_{(\R^n\times \R^n) \setminus (\mathcal C\Omega\times\mathcal C\Omega)} \frac{(u(x)-u(y))^2}{|x-y|^{n+2s}}\,dxdy\,,
\end{equation}
with $\mathcal C\Omega = \R^n \setminus \Omega$, and, for $s=1$,
\begin{equation}\label{energyfunct1s=1}
\cK^1(u,\Omega) = \frac 1 2 \int_\Omega |\nabla u|^2 \,dx.
\end{equation}

In this paper we establish an extension problem for the operator $L$. As a main application
we obtain the following 1-D symmetry result for layer solutions to \eqref{problema1}. Here  we assume $\mu(\{1\})=0$ since this will simplify significantly the notation throughout the paper. However the result also holds for $\mu(\{1\})>0$ with the same proof ---see Remark \ref{Rk1}. In the case of equation \eqref{problema0}, $\mu(\{1\})=0$
translates as $s_K<1$.

\begin{thm}\label{thm:1dsym}
Assume that $\mu(\{1\})=0$ and that $u\in L^{\infty}(\R^n)$ is a layer solution of \eqref{problema1}, that is, a solution satisfying \eqref{layer}. Assume that either $n=2$ and $s_*>0$, or that $n=3$ and $s_* \ge 1/2$, where $s_*$ is given by \eqref{def-sstar}.

Then, $u$ has 1-D symmetry. That is, $u(x)=u_0(a\cdot x)$ where $u_0:\R\rightarrow\R$ is a layer solution in dimension one of $L u_0 + W'(u_0)=0$ in $\R$ and $a\in \R^n$ is some unit vector.
\end{thm}

A proof of the existence of a 1-D layer solution is given in Section \ref{secExistence} for the case of even potentials $W$.

Theorem \ref{thm:1dsym} is clearly inspired in a conjecture of De Giorgi \cite{DG} for the Allen-Cahn equation: $-\Delta u = u-u^3$ in all $\R^n$. This conjecture states that, if $n\le 8$, then solutions $u$ which are monotone in one variable must have 1-D symmetry.  This has been proved in dimensions $n=2$ by Ghoussoub and Gui \cite{GG}, $n=3$ by Ambrosio and Cabr\'e \cite{AC}, and for $4\le n\le 8$, when one assumes in addition that $u$ is a layer solution, by Savin \cite{S}.

For the related nonlocal equation, $(-\Delta)^s u + W'(u)=0$ in all $\R^n$, analog results have been found for $n=2$ and $s=1/2$ by Cabr\'e and Sol\`a-Morales \cite{C-SM}, for $n=2$ and $s\in(0,1)$
by Cabr\'e and Sire \cite{C-Si1,C-Si2}, and for $n=3$ and $s\in [1/2,1)$ by Cabr\'e and Cinti \cite{C-Cinti1,C-Cinti2}.

In this paper, we show how several arguments in \cite{GG,AC,C-SM,C-Cinti1,C-Cinti2, C-Si1} can be adapted to equation \eqref{problema1} to obtain 1-D symmetry results. In these papers, symmetry is deduced from a Liouville type theorem. Provided that $u$ satisfies certain energy estimates, this Liouville type theorem implies that any two directional derivatives of $u$ coincide up to a multiplicative constant. This is equivalent to the 1-D symmetry.
At the time of completing this article, all the known symmetry results for the simplest equation $(-\Delta)^su+W'(u)=0$ \cite{C-Cinti1,C-Cinti2, C-Si1, C-SM} are proven using the extension problem of Caffarelli and Silvestre \cite{CS}, which is used in the previous references to state and prove the Liouville theorem. The main novelty of the present paper is that we have a non scale invariant operator and the existence of an extension problem is a priori unclear. Here, we find the natural  extension problem, and how one can prove the symmetry result using it.
This new extension problem, discussed in Section \ref{sec6}, consists of a ``system'' of (possibly infinitely many) singular elliptic PDEs which are coupled by a single Neumann type boundary condition and a common trace constraint.

The  ideas of this paper could be useful in other contexts where an extension operator is known for a family of operators and one needs to consider also sums (or integrals) of these operators.

A crucial step towards the 1-D symmetry consists of establishing a sharp estimate for the energy of monotone solutions in a ball of radius $R\ge 2$. Let us define
\begin{equation}\label{eq_phins}
\Phi_{n,s}(R) =
\begin{cases}
R^{n-1}(R^{1-2s}-1)(1-2s)^{-1}\quad &\mbox{if } s\neq 1/2\,,\\
R^{n-1}\log R& \mbox{if } s=1/2\,.\\
\end{cases}
\end{equation}
A useful property of $\Phi_{n,s}(R)$ is that it is continuous and decreasing in $s$  for all $R> 2$.

The following result is proven in Section \ref{sec4}. Throughout the paper we use the notation $B_R=\{x\in\R^n, \ |x|<R\}$.
\begin{prop}\label{thm:eg-est}
Let $u$ be a layer solution of \eqref{problema1}, i.e., a solution satisfying \eqref{layer}. Then, for all $R\ge2$,
\[ \cE(u,B_R) \le C\Phi_{n,s_*}(R),\]
where $s_*$ is given by \eqref{def-sstar}, $\Phi_{n,s}$ is given by \eqref{eq_phins}, and $C$ depends only on $n$, $s_*$ and $W$.
\end{prop}
Related energy estimates for the pure fractional Laplacian ---also for minimizers--- have been obtained in \cite{AC,C-Cinti2, C-Si1,SV2}.

An strategy to prove Proposition \ref{thm:eg-est} could be to show first that layer solutions are minimizers of the energy $\cE$ in every ball and to compare the energy of $u$ with some explicit competitor. This was done by Savin and
Valdinoci in \cite{SV2} for $L=(-\Delta)^s$ and their proof (with minor modifications) would give also the correct
energy estimate for minimizers of our energy $\mathcal E$. However, this requires to prove that layer solutions are minimizers via the standard foliation argument from \cite{AAC}.
For this, one needs regularity estimates for solutions to \eqref{problema1} in {\em bounded domains}. 
These estimates, in bounded domains and for general $L$ of the form \eqref{int_lap}, turn out to be true (see comments below) but more intricate than the estimate in the whole space, given by Proposition \ref{prop-regularity-mzrs} below. By this reason, to prove the energy estimate of Proposition \ref{thm:eg-est}, we follow a different approach \`a la  Ambrosio-Cabr\'e \cite{AC}, which allows to obtain  the estimate (for layers, not for minimizers) in a more straight-forward way.

Although for simplicity we only prove energy estimates and 1-D symmetry for layer solutions, our proofs can be adapted to the setting of minimizers. As said above, the energy estimate for minimizers can be obtained with the same comparison argument as in \cite{SV2}. Then, all the other proofs in the paper apply to minimizers up to standard modifications ---see \cite{C-Cinti2}.

Let us be more precise about the regularity issues commented above. First note that when $L$ is a finite sum of fractional
Laplacians as in \eqref{problema0}, then $L$ equals $(-\Delta)^{s_K}$ plus lower order operators, and thus  by
totally standard arguments one obtains optimal interior regularity estimates for $Lu=g$,  in all of $\R^n$ and in bounded
domains.
For infinite sums or integrals the situation  may be not as simple. To see it, consider the example
\[ \mathcal{L} u := \sum_{k\ge 1}\frac{1}{2^k} (-\Delta)^{\beta-\frac{1}{2k}}u\]
for some $\beta\in(1/2,1]$. This operator is ``almost of order $2\beta$'',
but does not have a definite order. When $\beta<1$ the operator $\mathcal{L}$ is ``nonlocal at every scale'', 
and the method of Silvestre \cite{Sil2} gives a $C^\alpha$ estimate for solutions to $\mathcal{L}u =g$ in the bounded domain. This approach does not work, however, when $\beta=1$, since the operator ``degenerates'' to a local one at infinitesimal scales.

The proof of interior estimates in the bounded domain for $\mathcal{L}$ when $\beta=1$,
even if the operator is translation invariant, is not completely standard. Two different ways of obtaining these interior estimates
are the following. One is proving first a good enough  estimate in the whole space $\R^n$ (say $C^{1,\alpha}$) and deducing from it the estimate in the bounded domain. This estimate in the whole space can be obtained from sharp heat kernel estimates like the ones in \cite{Mimica}. A second possibility is to use a blow-up and compactness argument as in \cite{Serra,RosOton-Serra}, deducing interior regularity from a Liouville type theorem for entire solutions with certain growth properties.

Let us now quickly link the energy functional $\cE$ with problem \eqref{problema1} and make precise our notion of solution to \eqref{problema1}. The quadratic form $\cK(\cdot,\Omega)$ comes from a scalar product, which we denote by $\langle \cdot,\cdot\rangle_{\Omega}$. Namely,
\begin{equation}\label{eq:scalar_prod1}
\cK(u,\Omega)=\frac{1}{2}\langle u,u\rangle_{\Omega}\,.
\end{equation}
This scalar product is defined by
\begin{equation}\label{eq:scalar_prod2}
\langle u,v\rangle_{\Omega}= \littint \langle u,v\rangle_{\Omega,s} \,d\mu(s),
\end{equation}
where
\begin{equation}\label{eq:scalar_prod3}
\langle u,v\rangle_{\Omega,s} = \frac{c_n(s)}{2} \iint_{(\R^n\times \R^n) \setminus (\mathcal C\Omega\times\mathcal C\Omega)} \frac{(u(x)-u(y))(v(x)-v(y))}{|x-y|^{n+2s}}\,dxdy\,.
\end{equation}

Minimizers of $\cE$ (with respect to compact perturbations) are functions $u: \R^n \rightarrow \R$ that satisfy, for every
bounded domain $\Omega$, $\varepsilon >0$, and $\xi\in C_c^\infty(\Omega)$,
\[\begin{split}
\cE(u, \Omega)&\le \cE(u+\varepsilon\xi, \Omega) \\
&=\cK(u,\Omega) + \varepsilon^2 \cK (\xi,\Omega)+ \varepsilon \langle u,\xi\rangle_{\Omega} +\int_{\Omega} W(u+\varepsilon\xi)\,dx.
\end{split}
\]
Equivalently,
\[0\le  \varepsilon \cK (\xi,\Omega)+ \langle u,\xi\rangle_{\Omega} + \int_{\Omega} \frac{1}{\varepsilon}(W(u+\varepsilon\xi)-W(u))\,dx\]
for every bounded domain $\Omega\subset\R^n$, $\varepsilon>0$, and $\xi\in C^{\infty}_{c}(\Omega)$. Letting $\varepsilon\searrow 0$, we obtain
\begin{equation}\label{weakvs}
\langle u,\xi\rangle_{\Omega}+\int_\Omega W'(u)\xi\,dx = 0\quad \mbox{for every } \Omega\subset\subset \R^n \mbox{ and } \xi\in C^{\infty}_{c}(\Omega)\,.
\end{equation}
Equation \eqref{weakvs} is the weak version of \eqref{problema1}. We will say that a function $u\in L^{\infty}(\R^n)$ is a {\em weak solution} of \eqref{problema1} if $\cE(u,\Omega)<\infty$ and \eqref{weakvs} is satisfied for all $\Omega\subset\subset\R^n$ and $\xi\in C^{\infty}_{c}(\Omega)$.

The relation between the weak and the strong formulations of the problem is given by the integration by parts type formula
\begin{equation}\label{int_by_parts_formula}
\langle u,v\rangle_{\Omega}= \int_{\Omega} L u(x) v(x)\,dx + \littint d\mu(s)  \,c_n(s) \int_{\cC\Omega}dx\int_{\Omega}dy\frac{u(x)-u(y)}{|x-y|^{n+2s}} v(x)\,,
\end{equation}
that holds for $u,v\in C^2(\R^n)$ bounded. This formula is found integrating with respect to $\,d\mu(s)$ the well-known identities
\begin{equation}\label{int_by_parts_formulas}
\langle u,v\rangle_{\Omega,s}= \int_{\Omega} (-\Delta)^s u(x) v(x)\,dx + c_n(s) \int_{\cC\Omega}dx\int_{\Omega}dy\frac{u(x)-u(y)}{|x-y|^{n+2s}} v(x)\,.
\end{equation}
These identities are very elementary but useful, for instance in our proof in Section~\ref{sec4} of the energy estimate for monotone solutions. Note the last term on the right side can be interpreted as a nonlocal flux.
The identity \eqref{int_by_parts_formulas} is easily proven by writing  $(-\Delta)^s u$ as a singular integral and rearranging some terms. One needs only to observe that  \[\text{PV}\int_\Omega \,dx\int_\Omega\,dy \frac{(u(x)-u(y))v(x)}{|x-y|^{n+2s}} = -\text{PV}\int_\Omega \,dx\int_\Omega\,dy \frac{(u(x)-u(y))v(y)}{|x-y|^{n+2s}}.\]

On the one hand, using the integration by parts formula \eqref{int_by_parts_formula} in  \eqref{weakvs} we find that, when $u$ is a smooth enough weak solution, we have
\[
\int_{\Omega} Lu\  \xi\,dx +\int_\Omega W'(u)\xi\,dx = 0\quad \mbox{for every } \Omega\subset\subset \R^n \mbox{ and } \xi\in C^{\infty}_{c}(\Omega),
\]
and hence $u$ is a solution of \eqref{problema1}.

On the other hand, if $u$ is merely a measurable function $u:\R^n\rightarrow[-1,1]$ we can also give a notion of solution to \eqref{problema1}, now integrating
by parts in the opposite direction.  Since $\xi\in C^{\infty}_{c}(\Omega)$ in \eqref{weakvs}, we find that $\langle u,\xi\rangle_{\Omega}= \int_{\R^n} uL\xi\,dx$
and thus
\begin{equation}\label{veryweakvs}
\int_{\R^n}  u L\xi\,dx +\int_{\R^n}  W'(u)\xi\,dx = 0\quad \mbox{for every } \xi\in C^{\infty}_{c}(\Omega)\,.
\end{equation}
This is the notion of solution to \eqref{problema1} {\em in the sense of distributions}. Every weak solution is also a solution in the sense of distributions.

Next proposition concerns $C^{2,\gamma}$ regularity of weak solutions to \eqref{problema1}. It is proved in Section \ref{secreg} using a result of Silvestre \cite{Sil2}.
In fact, we prove regularity not only for weak solutions but also for solutions of the equation in the whole $\R^n$ in the sense of distributions.
\begin{prop}\label{prop-regularity-mzrs}
Let $u\in L^{\infty}(\R^n)$ with $|u|\le 1$ in all $\R^n$. Assume that $u$ satisfies  \eqref{veryweakvs} with $\Omega=\R^n$. Then, $u\in C^{2,\gamma}(\R^n)$ and
\[\|u\|_{C^{2,\gamma}(\R^n)}\le C\]
for some $\gamma>0$ and $C$ depending only on $n$, $s_*$ and $W$.
\end{prop}

According to Proposition \ref{prop-regularity-mzrs}, layer solutions always satisfy equation \eqref{problema1} in the classical sense. Indeed, recall the well-known estimate (see the proofs of Theorems 2.5, 2.6 and 2.7 in \cite{Sil})
\begin{equation}\label{cotinuousoperator:C2gammatoCgamma}
\|(-\Delta)^s u\|_{C^{0,\gamma}(\R^n)} \le C \|u\|_{C^{2,\gamma}(\R^n)}\,,
\end{equation}
for every $u\in C^{2,\gamma}(\R^n)$, with $C$ uniform for $s\in[s_*,1)$  (depending only on $n$ and $s_*$).  Then, since $\mu$ is a probability measure,  $L u = \littint (-\Delta)^s u \,d\mu(s)$ is still in $C^{0,\gamma}(\R^n)$ and thus the equation is satisfied in the ``classical sense''.

The paper is organized as follows:
In Section \ref{secreg} we prove the regularity Proposition \ref{prop-regularity-mzrs}. In Section \ref{secExistence} we show the existence of 1-D layer solutions in the case of even potentials $W$.  In Section \ref{sec4} we prove
the energy estimate of Proposition \ref{thm:eg-est}. In Section \ref{sec6} we introduce the extension problem for the operator $L$ that allows us to reformulate problem \eqref{problema1} as a system of PDEs.
In Section \ref{sec7}, the last one,  we obtain a Liouville type theorem within the framework of the extension problem and we us it to prove the 1-D symmetry result, Theorem \ref{thm:1dsym}.

\section{Regularity} \label{secreg}

In this section we prove Proposition \ref{prop-regularity-mzrs}. It will be obtained by iterating the following

\begin{lem}\label{lem-reg}
Let $u\in L^\infty(\R^n)$ satisfy $L u=w$ in all of $\R^n$ in the sense of distributions. Assume that $w\in C^\beta(\R^n)$, $\beta\ge0$. Then, there exist $\alpha>0$ and $C$ depending only on $n$ and $s_*$ such that $u\in C^{\beta+\alpha}\bigl(\R^n\bigr)$ and
\[ \|u\|_{C^{\beta+\alpha}(\R^n)} \le C\bigl( \|u\|_{L^\infty(\R^n)}+ \|w\|_{C^\beta(\R^n)}\bigr)\]
\end{lem}
\begin{proof}
Since $L$ is linear and translation invariant, it commutes with convolution. Thus, by considering convolutions of $u$ and $w$ with a smooth approximation of the identity,  we may assume that $u$ and $w$ are smooth and that the equation holds in strong sense.

Let us consider first the case $\beta=0$.
Let $\epsilon=s_*/2$ and $v=(-\Delta)^{\epsilon} u$.
Then $v$ satisfies
\[\tilde L v = w\quad \mbox{in } \R^n,\]
where $\tilde L =\littint_{[\epsilon, 1-\epsilon]} d\mu(\epsilon+t)(-\Delta)^t$. Since $\tilde L$ is a convex combination of
fractional Laplacians $(-\Delta)^t$ with $t\in [\epsilon, 1-\epsilon]$, the results of Silvestre in \cite{Sil2} apply
to $\tilde L$ (but not to $L$). More precisely, see the proof of Theorem 5.4, Remark 4.3, Proposition 3.1, and Sections 3.1 and 3.2 of \cite{Sil2}. We obtain
\begin{equation} \label{silvestre}
\|v\|_{C^{\bar\alpha}(\R^n)} \le C\bigl( \|v\|_{L^\infty(\R^n)}+ \|w\|_{L^\infty(\R^n)}\bigr),
\end{equation}
where ${\bar\alpha}$ and $C$ depend only on $n$ and $s_*$ (we are using that $\mu$ is a probability measure).

But by classical  Riesz potential estimates \cite{Landkov}, since $(-\Delta)^\epsilon u=v$, we have
\begin{equation}\label{quinnom}
\|u\|_{C^{{\bar\alpha} +2\epsilon}(\R^n)} \le C\bigl( \|u\|_{L^\infty(\R^n)}  + \|v\|_{C^{\bar\alpha}(\R^n)} \bigr)
\end{equation}
and, since $\bar\alpha/2+2\epsilon>2\epsilon$,
\begin{equation}\label{quinnom2}
 \|v\|_{L^\infty(\R^n)}  \le  C \|u\|_{C^{\bar\alpha/2+ 2\epsilon}(\R^n)}.
 \end{equation}

Therefore it follows from \eqref{silvestre}, \eqref{quinnom}, and \eqref{quinnom2} that
\[ \begin{split}
 \|u\|_{C^{{\bar\alpha}+2\epsilon}(\R^n)} &\le C\bigl( \|u\|_{L^\infty(\R^n)}  + \|v\|_{L^\infty(\R^n)} +  \|w\|_{L^\infty(\R^n)} \bigr)
 \\ &\le C \bigl( \|u\|_{C^{{\bar\alpha}/2+2\epsilon} (\R^n)}+ \|w\|_{L^\infty(\R^n)} \bigr) .
 \end{split}
 \]
Thus, the estimate of the lemma with $\alpha = \bar\alpha+2\epsilon$  and $\beta=0$ follows using a standard interpolation inequality.

The cases $\beta>0$ follow applying the previous case to incremental quotients (of derivatives if $\beta>1$) of $u$ and $w$.
\end{proof}

Finally, we prove Proposition \ref{prop-regularity-mzrs}.
\begin{proof}[Proof of Proposition \ref{prop-regularity-mzrs}]

Since $u$ and $W'(u)$ belong to $L^{\infty}(\R^n)$, Lemma \ref{lem-reg} applied with $\beta=0$ yields the bound
$\|u\|_{C^{\alpha}(\R^n)}\le C$ for some $\alpha$ depending only on $n$ and $s_*$,  and some $C$ depending on $n$, $s_*$, $W$. But $W'$ is a $C^2$ function and hence we find also a bound for $\|W'(u)\|_{C^{\alpha}(\R^n)}$. This starts a standard  bootstrap argument that leads, after using Lemma \ref{lem-reg} $\lceil 2/\alpha\rceil$ times, to the estimate $\|u\|_{C^{2,\gamma}(\R^n)}\le C$, where $\gamma= \lceil 2/\alpha\rceil\alpha-2$.
\end{proof}

\section{Existence of layer solutions}\label{secExistence}

In this section we prove that there exists a layer solution to $Lu + W'(u)=0$ in $\R$ in the case that $W$ is even.
Uniqueness of layer solution in $\R$ holds in case $W''(\pm 1)>0$. We do not present here the details of the uniqueness proof.
It is based in the sliding method, as in \cite{C-SM, C-Si2}.

Our existence result relies on the a priori estimates proved in the previous section and on the following proposition on existence of a layer solution for the modified operator
\begin{equation}\label{Ldelta}
L_\delta = \delta(-\Delta) + (1-\delta)L,
\end{equation}
$\delta\in(0,1)$.
For this, we consider the energy functional
\[ \mathcal E_\delta\bigl(u, (-R,R)\bigr) = \frac{\delta}{2} \int_{-R}^R |\nabla u|^2 + (1-\delta) \mathcal K\bigl(u, (-R,R)\bigr) + \int_{-R}^R W(u) \,dx.\]
The gain in considering $L_\delta$ instead of $L$ is that the new operator is the (minus) Laplacian plus lower order terms, and hence it has interior regularity
estimates inherited from those of the Laplacian.

\begin{prop}\label{propdelta}
Let $L_\delta$ be defined by \eqref{Ldelta}. Then, if $W$ is even, i.e. $W(-t)=W(t)$, there exists a  bounded odd solution $L_\delta u + W'(u)=0$ in $\R$  satisfying $|u|\le 1$ with $u(0)=0$ and $\lim_{x\to \pm \infty} u(x) = \pm 1$. Moreover, $u\in C^{2,\alpha}(\R)$ for some $\alpha\in(0,1)$ and it is a minimizer of $\mathcal E_\delta$ with respect to smooth compactly supported perturbations $\xi$ with $|\xi|\le 1$.
\end{prop}
\begin{proof}
The proof of the lemma exploits the fact that $L_\delta$ is the minus Laplacian plus lower order terms. The existence of solution will follow using
a rather standard approach ---see for instance \cite{C-SM, PSV} for related proofs. We divide the proof in four steps.

{\em Step 1.} For any given $R>0$, we prove the existence of a strong  solution $u_R$ to
\begin{equation}\label{problemaR}
\begin{cases}
L_\delta u_R + W'(u_R) = 0 \quad &\mbox{in }(-R,R)\\
u_R = -1  &\mbox{in }(-\infty,-R)\\
u_R = 1  &\mbox{in }(R,+\infty).
\end{cases}
\end{equation}

Let $\varphi$ be defined as
\[
\varphi(x)=
\begin{cases}
x/R \quad &\mbox{in }(-R,R)\\
 -1  &\mbox{in }(-\infty,-R)\\
 1  &\mbox{in }(R,+\infty).
\end{cases}
\]
Let $\bar v$ be a minimizer of the energy functional $\mathcal E_\delta$ among all functions $v$ in
\[ X = \{ v\in \varphi + W^{1,2}_0\bigl((-R,R)\bigr) \,:\  |v|\le 2\}.\]
Note that here the functions $v$ are defined in all of $\R$ and $v \equiv 1$ in $[R,\infty)$, $v \equiv -1$ in $(-\infty,R]$.

The existence of $\bar v$ is proved by the ``direct method in the calculus of variations'', using that
\[\frac 1 2 \int_{\R} | v'|^2\,dx = \frac 1 2 \int_{-R}^R |\nabla v|^2\,dx \le \delta^{-1} \mathcal E_\delta(v)\]
and that for $x_1<x_2$
\[  |v(x_2)-v(x_1)|\le \int_{x_1}^{x_2} |v'|\,dx \le \left(\int_\R |v'|^2\,dx\right)^{1/2} |x_2-x_1|^{1/2}, \]
to obtain the compactness of a minimizing sequence. The lower semicontinuity of the $(1-\delta)\mathcal K$ term of $\mathcal E_\delta$ is by now standard and follows easily from Fatou's lemma. The lower semicontinuity of two remaining terms of $\mathcal E_\delta$ is classical.

We  observe that $\bar v\le 1$ in $\R$ since otherwise the function $\min\{\bar v,1\}$ would have strictly less energy than $\bar v$. Similarly, $\bar v \ge -1$. Since in the definition of the minimization space $X$ we constrain $|v|\le 2$, and we have shown that $|\bar v| \le 1$, the minimizer $\bar v$ is a weak solution to
\eqref{problemaR}.

Since $L_\delta$ is the minus Laplacian plus lower order nonlocal terms whose kernels have smooth tails, we show next that $\bar v\in C^{2,1/2}_{\rm loc}\bigl((-R,R)\bigr)$ and hence $\bar v$ is a strong solution to \eqref{problemaR}. Indeed, we have $\bar v\in C^{1/2}(\R)$ simply from the embedding $W^{1,2}(\R) \subset C^{1/2}(\R)$. Then, the function $\bar v$ is a distributional solution to
\[  -\delta \Delta \bar v = -W'(\bar v) - (1-\delta)L\bar v \qquad\text{in } (-R,R).\]

Let $r_0\in (0,1)$ be a small constant to be chosen later. Given $x_0\in (-R,R)$ and $r\in (0,r_0)$ such that
$(x_0-4r, x_0+4r)\subset (-R,R)$, the rescaled function $w = \bar v(x_0- r\,\cdot\,)$ satisfies
\[  -\delta w'' = -\delta \Delta w = -  r^2W'(w) -  (1-\delta)\littint_{[s_*,1)} r^{2-2s}(-\Delta)^s  w\,d\mu(s) \quad
\mbox{ in }(-4,4).\]
Thus, using that $|w|\le 1$ and thus $|W'(w)|\le C$ and,  integrating the ODE two times we obtain
\begin{equation}\label{newnew}
[w]_{C^{2,1/2}([-1,1])} \le  \frac 1\delta \left( C + \int_{[s_*,1)} r^{2-2s}\|(-\Delta)^s
w\|_{C^{0,1/2}([-1,1])}\,d\mu(s)\right)
\end{equation}
where $C$ depends only on $W$.
But by a standard estimate, using again that $|w|\le 1$ in all of $\R$,
\[ \|(-\Delta)^s w\|_{C^{0,1/2}([-1,1])} \le   C (1+  [ w ]_{C^{2,1/2}([-2,2])}) ,\]
with $C$ depending only on $s_*$.

Hence, \eqref{newnew} yields
\[
[w]_{C^{2,1/2}([-1,1])} \le \frac{C}{\delta} + \frac{C}{\delta}\bigl( \littint_{[s_*,1)} (r_0)^{2-2s}\,d\mu(s)\bigr) [w]_{C^{2,1/2}([-2,2])},
\]
where we have used that $r<r_0$.

Scaling back the previous estimate from $w$ to $\bar v$ we obtain
\[  r^{5/2}[\bar v]_{C^{2,1/2}(x_0-r, x_0+r)} \le C/\delta + \rho(r_0) r^{5/2}[\bar v]_{C^{2,1/2}(x_0-2r, x_0+2r)} \]
for all $x_0\in (-R,R)$ and $r\in (0,r_0)$ such that $(x_0-4r, x_0+4r)\subset(-R,R)$,
where
\[\rho(t) :=\frac{C}{\delta}\bigl( \littint_{[s_*,1)}t ^{2-2s}\,d\mu(s)\bigr)\]
is some modulus of continuity (that is $\rho(t) \searrow 0$ as $t\searrow 0$) depending only on $\delta$ and $\mu$.

Then, it follows from $|\bar v|\le 1$ and the interpolation inequality for adimensional H\"older seminorms  \cite{GT}[Lemma 6.32 in Section 6.8] that
\[ r^{5/2}[\bar v]_{C^{2,1/2}(-R+r,R-r)} \le C,\]
for all $r>0$, where $C$ depends only on $\delta$, $\mu$, and $W$.

We define $u_R := \bar v$.

{\em Step 2.} We next show that $u_R$ is nondecreasing and odd.
This follows using the sliding technique.
Namely, define for $t>0$ the function $\tilde u^t(x) = u_R(x+t)$. When $t$ is large $\tilde u^t$ stays above $u_R$. Since
there can not be  contact points in $(-R,R-t)$ between the two strong solutions of the same equation $u_R$ and $\tilde u^t$,
we conclude that we can keep continuously decreasing $t$, preserving the inequality $\tilde u^t \ge u_R$ in $\R$ until we reach $t=0$. This means that $u_R$ is nondecreasing.
A similar procedure done now with $-u_R \bigl(-(x+t)\bigr)$ instead of $\tilde u^t$, which is also a solution to the same equation since $W$ is even (and hence $W'$ odd),  shows that $u_R$ is odd and in particular $u_R(0)=0$.

{\em Step 3.}  We let $R_k\to +\infty$ and show next that the previous minimizers $u_{R_k}$ in $(-R_k,R_k)$ converge (up to a subsequence) to an odd and nondecreasing solution of
\[
L_\delta u + W'(u) = 0 \quad \mbox{in }\R\\
\]
which minimizes the energy with respect to compactly supported perturbations $\xi\in C^\infty_c(\R)$ with $|\xi|\le 1$.

Indeed, for given  $\alpha'\in (0,1/2)$, the convergence in $C^{2,\alpha'}_{\rm loc}(\R)$  of a subsequence of $u_{R_k}$ to a solution strong solution $u$ in $\R$ is standard and follows from the interior $C^{2,1/2}$ estimates (which are independent of $R$) and the Arzel\`a-Ascoli theorem. The limiting function $u$ is odd and nondecreasing since  $u_{R_k}$ are so.

The fact that $u$ is a minimizer with respect to compactly supported perturbations follows by passing to the limit in
the inequality
\[
\begin{split}
 &\mathcal E_\delta\bigl( u_{R_k}+ \xi , (-R, R)\bigr) - \mathcal E_\delta\bigl( u_{R_k} , (-R, R)\bigr) =\\
& \hspace{20mm}= \mathcal E_\delta\bigl( u_{R_k}+ \xi , (-R_k, R_k)\bigr) - \mathcal E_\delta\bigl( u_{R_k} , (-R_k, R_k)\bigr) \ge 0,
\end{split} \]
which holds for all $\xi\in C^\infty_c\bigl((-R,R)\bigr)$ with $|\xi|\le 1$, and $R_k>R$.

{\em Step 4.} It remains to show that  the solution $u$ build in Step 3  satisfies the limits $\lim_{x\to \pm \infty} u(x) = \pm 1$. Otherwise it could be the trivial solution $u \equiv 0$.
This follows from the minimality property of $u$. Indeed, let $\ell = \lim_{x\to +\infty} u(x)$ --- recall that $u$ is nondecreasing. If it were $\ell<1$ then
we would have
\[\mathcal E_\delta \bigl(u, (-R,R)\bigr) \ge \int_{-R}^R W(u) dx \ge cR\]
for some $c>0$. But then it is easy to build a competitor $w$ with $w= u$ outside $(-R,R)$, $|w-u|\le 1$, and  satisfying
\begin{equation}\label{estimatew}
\mathcal E_\delta \bigl(w, (-R,R)\bigr) \le C R^{1-2s_*},
\end{equation}
which would contradict the minimality of $u$ when taking $R$ large enough.
Indeed, the competitor $w$ is simply defined by $w=\max\{u, \psi\}$ where
\[
\psi(x) =
\begin{cases}
1 \quad &\mbox{if }|x|\le R-2\\
R-1-|x| &\mbox{if }R-2\le |x|\le R\\
-1 & \mbox{if }R\le |x|.
\end{cases}
\]
With this choice of $w$, \eqref{estimatew} is established by straightforward computation using that $|w'|\le C$ in all of $\R$ ---since $|\psi'| \le 1$ and $|u'|\le C$ by the previous estimates for $u_R$ in this same proof---,
the $n=1$ case of Claim \ref{claim:energy} in next section, that $w\equiv 1$ in $B_{R-2}$, and that $B_R\setminus B_{R-2}$ has length 4.
\end{proof}

We obtain the following
\begin{cor}
Given $W$ even, i.e. $W(-t)=W(t)$, there exists a unique bounded  solution $L u + W'(u)=0$ in $\R$  satisfying $|u|\le 1$ with $u(0)=0$ and $\lim_{x\to \pm \infty} u(x) = \pm 1$. In particular $u$ is odd. Moreover, $u\in C^{2,\alpha}(\R)$ and it is a minimizer of $\mathcal E$ with respect to smooth compactly supported perturbations $\xi$ with $|\xi|\le 1$.
\end{cor}

\begin{proof}
We will build the solution $u$ by considering the solution $u_\delta$ to  $L_\delta u_\delta + W'(u_\delta)=0$ in $\R$ from Proposition \ref{propdelta} and sending $\delta\to 0$.

The crucial observation that makes possible this approach is that the ``a priori'' estimate of Proposition \ref{prop-regularity-mzrs}, that has been obtained through the factorization trick, is also true (with identical proof) for the modified operator $L_\delta$.
Indeed, the factorization trick exploits the fact that $L= \tilde L \circ (-\Delta)^\epsilon $ where $\tilde L$ is an operator of the form  $\int_{[\epsilon, 1-\epsilon]} (-\Delta)^t \,\tilde \mu(dt)$ to which Silvestre's $C^\alpha$ estimate applies because it remains nonlocal at every small scale.
It is clear that the same type of factorization holds for $L_\delta$ and hence the proofs of section \ref{secreg} apply to $u_\delta$ without any change to yield
\[  \|u_\delta\|_{C^{2,\gamma} (\R)} \le C \]
with $C$ depending only on $s_*$ and $W$ (but not on $\delta$).

Therefore, letting $\delta \to 0$ the functions $u_\delta$ converge (up to a subsequence) in $C^{2+\gamma'}_{\rm loc} (\R)$ for all $\gamma'\in(0,\gamma)$, to a monotone odd solution to
$L u + W'(u)=0$ in $\R$.  
Moreover, similarly as in the proof of Proposition \ref{propdelta}, we can pass to the limit in the inequality
\[ \mathcal E_\delta\bigl( u_{\delta}+ \xi , (-R, R)\bigr) \ge \mathcal E_\delta\bigl( u_{\delta} , (-R, R)\bigr) \]
whenever $R>0$ and $\xi\in C^\infty_c\bigl((-R,R)\bigr)$ with $|\xi|\le 1$ to obtain an analog minimality property for $u$. Then, with the same energy comparison strategy as in the proof of Proposition \ref{propdelta} we rule out the possibility $\lim_{x\to +\infty} u(x) <1 $ and thus $\lim_{x\to \pm \infty} u(x) = \pm 1$.
\end{proof}

\section{Energy estimates}\label{sec4}

In this section we establish the energy estimate of Proposition \ref{thm:eg-est} for layer solutions of \eqref{problema1}.

Next Claim will be used to prove the energy estimates. Recall the definition of $\Phi_{n,s}(R)$ from \eqref{eq_phins}. The proof of the claim 
is a simple calculation and it is given at the end of the section.
\begin{claim} \label{claim:energy}For every $R\ge 2$, we have
\[c_n(s)\int_{B_{R}} \int_{\cC B_{R} } \frac{\min\{1,|x-y|\}}{|x-y|^{n+2s}}\,dx \,dy\le C\Phi_{n,s}(R)\,\]
where $C$ depends only on $n$ (but not on $s$).
\end{claim}

The following proposition establishes the energy estimate for layer solutions in every dimension. Since there is no extra effort in doing it, we prove a slightly more general statement that can be used to show energy estimates  for monotone solutions (without limits) in dimension three, as in Section 6 of \cite{C-Cinti2}.

\begin{prop}\label{prop_3}
Let $u$ be a solution of \eqref{problema1} which is monotone in the $x_n$ direction. Define $\overline u:\R^n\rightarrow \R$ by $\overline{u}(x',x_n) = \overline{u}(x')= \lim_{x_n\to +\infty} u(x',x_n)$. Then, there exists a constant $C$ depending only on $n$, $s_*$, and $W$, such that
\begin{equation}\label{eg_est_montne}
\cE(u,B_R)-\cE(\overline u,B_R)\le C \Phi_{n,s_*}(R)
\end{equation}
for every $R\ge 2$.

\begin{proof}
Consider, as in \cite{AC}, the slided function $u^t$, $t\ge0$, defined by $u^t(x',x_n) = u(x',x_n+t)$.

Using the integration by parts formula \eqref{int_by_parts_formula} and the equation satisfied by $u^t$ we find
\begin{equation}\label{eq:ddtE}
\frac{d}{dt} \cE(u^t,B_R) = \littint d\mu(s) c_n(s)\int_{\cC{B_R}}\negmedspace dx\int_{B_R}\negmedspace dy\frac{u^t(x)-u^t(y)}{|x-y|^{n+2s}} \partial_t u^t(x)\,.
\end{equation}
Indeed, we have
\[\begin{split}
\frac{d}{dt} \cE(u^t,B_R) &= \langle u^t, \partial_t u^t\rangle_{B_R}+ \int_\Omega W'(u) \partial_t u^t \,dx\\
&= \int_{B_R} L u^t \partial_t u^t +\littint d\mu(s)c_n(s)\int_{\cC{B_R}}\negmedspace dx\int_{B_R}\negmedspace dy\frac{u^t(x)-u^t(y)}{|x-y|^{n+2s}}  \partial_t u^t(x)\ +\\
&\qquad\qquad\qquad\qquad\qquad\qquad  +\int_{B_R} W'(u^t) \partial_t u^t (u) \,dx\,,
\end{split}
\]
and note that $L u^t + W'(u^t) \equiv 0$.

Using the bound $\|u^t\|_{C^{2,\gamma}(\R^n)}\le C$ in Proposition \ref{prop-regularity-mzrs} with $C$ depending only on $n$, $s_*$, and $W$ ---thus, $C$ independent of $t$--- we find, by monotone convergence, that $u^t\to \bar u$ in  $C^{2,\gamma}_{\text{loc}}(\R^n)$. We also find that $|u^t(x)-u^t(y)|\le C \min\{1,|x-y|\}$.

Therefore, we have
\[\cE(u,B_R)-\cE(\overline u,B_R) = \bigl.\cE(u^t,B_R)\bigr|_{+\infty}^0 = -\int_0^{+\infty} \frac{d}{dt} \cE(u^t,B_R)\,dt\,.\]
Integrating \eqref{eq:ddtE}, using that $\partial_t u^t = \partial_{x_n}u^t \ge 0$, and Claim \ref{claim:energy}, we obtain
\[\begin{split}
\cE(u,B_R)-\cE(\overline u,B_R)&= -\int_0^{+\infty} \frac{d}{dt} \cE(u^t,B_R)\,dt\\
&= - \int_0^{+\infty}\negmedspace dt \littint d\mu(s) c_n(s)\int_{\cC{B_R}}\negmedspace dx\int_{B_R}\negmedspace dy\frac{u^t(x)-u^t(y)}{|x-y|^{n+2s}} \partial_t u^t(x)\\
&\le \int_0^{+\infty}\negmedspace dt \littint d\mu(s) c_n(s)\int_{\cC{B_R}}\negmedspace dx\int_{B_R}\negmedspace dy\frac{C\min\{1,|x-y|\}}{|x-y|^{n+2s}} \partial_t u^t(x)\\
&= \littint d\mu(s) c_n(s)\int_{\cC{B_R}}\negmedspace dx\int_{B_R}\negmedspace dy\frac{C\min\{1,|x-y|\}}{|x-y|^{n+2s}} \int_0^{\infty}\partial_t u^t(x)\,dt\\
&\le C \|u\|_{L^{\infty}(\R^n)} \littint d\mu(s) c_n(s)\int_{\cC{B_R}}\negmedspace dx\int_{B_R}\negmedspace dy\frac{\min\{1,|x-y|\}}{|x-y|^{n+2s}}\\
&\le C  \littint d\mu(s)\Phi_{n,s}(R)\\
&\le C  \Phi_{n,s_*}(R)\,,
\end{split}
\]
for some $C$ depending only on $n$, $s_*$, and $W$. We have also used the fact that $\Phi_{n,s}$ is decreasing in $s$ for $R\ge 2$.
\end{proof}
\end{prop}

We give the
\begin{proof}[Proof of Proposition \ref{thm:eg-est}]
It is an immediate consequence of Proposition \ref{prop_3}, observing that, for layer solutions, we have $\overline u\equiv 1$  and clearly $\cE(1,B_R)=0$ for all $R>0$.
\end{proof}

We finally give the
\begin{proof}[Proof of Claim \ref{claim:energy}]
Observe that
\[\begin{split}\int_{B_{R}} \int_{\cC B_{R} } \frac{\min\{1,|x-y|\}}{|x-y|^{n+2s}}\,dx \,dy &\le
\int_{B_{R-1}} \int_{\cC B_{R} } \frac{\,dx \,dy}{|x-y|^{n+2s}} +\\
+\int_{B_{R}} \int_{\cC B_{R+1} } &\frac{\,dx \,dy}{|x-y|^{n+2s}}
+\int_{B_R\setminus B_{R-1}}\int_{B_{R+1\setminus B_R} } \frac{\,dx \,dy}{|x-y|^{n+2s-1}}\,.
\end{split}\]

The first term is bounded as follows: for $x\in B_{R-1}$ we have
\[\phi(x) := \int_{\cC B_{R} } \frac{\,dy}{|x-y|^{n+2s}} \le  \int_{R-|x|}^\infty \frac{r^{n-1}dr}{r^{n+2s}} = \frac{1}{2s} (R-|x|)^{-2s}\,.\]
Therefore,
\[\begin{split}
\int_{B_{R-1}}\int_{\cC B_{R} } \frac{\,dx \,dy}{|x-y|^{n+2s}} &= \int_{B_{R-1}} \phi(x)dx \le
 \frac{C}{s}\int_0^{R-1} \frac{r^{n-1}dr}{(R-r)^{2s}}\\
 &\le  \frac{CR^{n-1}}{s}\int_0^{R-1} \frac{dr}{(R-r)^{2s}} =\frac{C}{s}\Phi_{n,s}(R)\,,
\end{split}
\]
where $C$ denotes different constants throughout the proof which depend only on $n$.

The second term is identical, having $R+1$ instead of $R$. Thus, it is also bounded by $\frac{C}{s}\Phi_{n,s}(R)$.

The third term is easily bounded in dimension $n=1$ and we will use this later for general $n$. Indeed, if $s\neq1/2$,
\[\begin{split}
\int_{R-1}^R\int_{R}^{R+1}\frac{dx\,dy}{|x-y|^{2s}}&=\int_{-1}^{0}\int_{0}^1
\frac{dx\,dy}{(x-y)^{2s}} \\
&=\frac
{1}{1-2s}\int_{-1}^{0} \bigl((1-y)^{1-2s}-(-y)^{1-2s}\bigr)\,dy\\
&= \frac{1}{2(1-s)} \frac{2^{2-2s}-2}{1-2s}\\
&\le \frac{C}{1-s}.
\end{split}
\]
For $s=1/2$ we have
\[
\int_{R-1}^R\int_{R}^{R+1}\frac{dx\,dy}{|x-y|}=\int_{-1}^{0}\int_{0}^1
\frac{\,dx\,dy}{x-y}
=\int_{-1}^{0} \log\bigl(\textstyle\frac{1-y}{-y}\bigr)\,dy
\le C.
\]

It remains to bound the third term  for $n>1$.  We proceed as follows:
\[
\begin{split}
\int_{B_R\setminus B_{R-1}}\int_{B_{R+1\setminus B_R} } \frac{\,dx \,dy}{|x-y|^{n+2s-1}}&= \int_{R-1}^{R}\int_{R}^{R+1}\int_0^\pi\frac{C r_1^{n-1}r_2^{n-1} (\sin \theta)^{n-2} \,d\theta\,dr_1\,dr_2}{(r_1^2+r_2^2-2r_1r_2\cos\theta)^{\frac{n+2s-1}{2}}}\\
&\le\int_{R-1}^{R}\int_{R}^{R+1}\int_0^\pi\frac{CR^{n-2} r_1r_2 \,d\theta\,dr_1\,dr_2}{(r_1^2+r_2^2-2r_1r_2\cos\theta)^{\frac{2s+1}{2}}}\,,
\end{split}
\]
where we have used that, for all
$r_1$, $r_2$ and $\theta$ in the domain of integration, we have
\[\frac{r_1r_2 \sin \theta}{(r_1^2+r_2^2-2r_1r_2\cos\theta)^{\frac{1}{2}}}\le R\,.\]
This follows from $r_2\le R$ and the fact that $x\sin\theta\le (x^2+1-2x\cos\theta)^{1/2}$ for all $\theta$ and $x$ (where we take $x =r_1/r_2$).

Next, we bound (here we make the change $1-\cos\theta =t^2$)
\[\begin{split}
\int_0^\pi\frac{CR^{n-2} r_1r_2 \,d\theta}{(r_1^2+r_2^2-2r_1r_2\cos\theta)^{\frac{2s+1}{2}}} &=
\int_0^\pi\frac{CR^{n-2} r_1r_2 \,d\theta}{((r_1-r_2)^2+2r_1r_2(1-\cos\theta))^{\frac{2s+1}{2}}}
\\
&\le 2\int_0^{\pi/2}\frac{CR^{n-2} r_1r_2 \,d\theta}{((r_1-r_2)^2+2r_1r_2(1-\cos\theta))^{\frac{2s+1}{2}}}
\\
&\le\int_0^1\frac{CR^{n}  \,dt}{((r_1-r_2)^2+R^2t^2)^{\frac{2s+1}{2}}}
\\
&\le \frac{CR^{n}}{(r_1-r_2)^{2s+1}}\int_0^\infty
\frac{dt}{\bigl(1+\bigl(\frac{Rt}{r_1-r_2}\bigr)^2\bigr)^{\frac{2s+1}{2}}}\\
&=\frac{CR^{n-1}}{(r_1-r_2)^{2s}}\int_0^{\infty}
 \frac{d\xi}{(1+\xi^2)^{\frac{2s+1}{2}}}\\
&\le \frac{1}{s}\frac{CR^{n-1}}{(r_1-r_2)^{2s}}.
\end{split}
\]
We have thus come back to  the situation of  dimension $n=1$. Indeed, from the previous inequalities
\[
\begin{split}
\int_{B_R\setminus B_{R-1}}\int_{B_{R+1\setminus B_R} } \frac{\,dx \,dy}{|x-y|^{n+2s-1}}&\le
\frac{CR^{n-1}}{s}\int_{R-1}^{R}\int_{R}^{R+1}\frac{\,dr_1\,dr_2}{(r_1-r_2)^{2s}}\\
&\le
\frac{CR^{n-1}}{s(1-s)}\,,
\end{split}
\]
where $C$ depends only on $n$.

Putting together the bounds for the three terms, we have proved that
\[
\int_{B_{R}} \int_{\cC B_{R} } \frac{\min\{1,|x-y|\}}{|x-y|^{n+2s}}\,dx \,dy\le \frac{C}{s(1-s)}\Phi_{n,s}(R).
\]

Multiplying this inequality by $c_n(s)$ ---as in the statement of the claim--- and using that $\frac{c_n(s)}{s(1-s)}$ is uniformly bounded for $s\in[0,1)$ ---as it is immediate to check in \eqref{eq_c_n(s)}---, we conclude the proof.
\end{proof}

\section{Extension problem}\label{sec6}
In this section we give a local formulation of problem \eqref{problema1}:
\[\littint (-\Delta)^s u\, d\mu(s) = f(u)\,,\]
where we define $f= -W'$.
This can be done by working, at the same time, with several (or possibly infinitely many) extension problems of
Caffarelli-Silvestre type \cite{CS}.

Given $s\in(0,1)$ and $u\in L^{\infty}(\R^n)$, the {\em $s$-extension of $u$ to $\R^{n+1}_+$} is defined by
\[\tilde{u}_s(\cdot,\lambda)= P_s(\cdot,\lambda)\ast u \quad \mbox{in }\R^n\]
for all $\lambda>0$, where $P_s$ is
\[
 P_s(x,\lambda)= p_{n,s}\frac{\lambda^{2s}}{\bigl(|x|^2 +\lambda^2\bigr)^{\frac{n+2s}{2}}}\,,
\]
and $p_{n,s}$ is the constant for which $\int_{\R^n} P_s(x,\lambda)\,dx=1$.

The function $\tilde{u}_s$ solves the extension problem of Caffarelli and Silvestre \cite{CS}:
\begin{equation}
\left\{
\begin{array}{ll}
\nabla\cdot(\lambda^{1-2s}\nabla \tilde{u}_s)=0 \quad &\text{in }\R^{n+1}_+=\{(x,\lambda)\,,\  x\in\R^n\,, \ \lambda>0\}\,,\\
\tilde{u}_s(x,0) = u(x)  \quad&\text{on } \{\lambda=0\}\,.
\end{array}
\right.
\end{equation}
Moreover, from results in \cite{CS} we have that, for $u$ regular enough,
\[(-\Delta)^{s} u(x) = - d(s)\lim_{\lambda\to 0^+} \lambda^{1-2s} \partial_{\lambda} \tilde{u}_s (x,\lambda)\,,\]
where $d(s)$ is a constant depending only on $s$ (see also \cite{C-Si1}).

From the considerations above, to every solution of problem \eqref{problema1}, it corresponds a solution of the following system of PDEs:
\begin{equation}\label{system_extended}
\left\{
\begin{array}{ll}
\nabla\cdot(\lambda^{1-2s}\nabla \tilde{u}_s)=0 \quad &\text{in }\R^{n+1}_+\,,\\
\tilde{u}_s(x,0) = u(x)  \quad&\text{on } \{\lambda=0\}\,,\\
- \littint d\mu(s) \,d(s)\lim_{\lambda\to0^+} \lambda^{1-2s} \partial_{\lambda} \tilde{u}_s (x,\lambda)  = f(u)&\text{on } \{\lambda=0\}.
\end{array}
\right.
\end{equation}
This system possibly involves infinitely many unknowns, the functions $\{\tilde{u}_s\}_{s\in{\rm supp}\,\mu}$. Note that if we consider the operator $L= \sum_{i=1}^K \mu_i(-\Delta)^{s_i}$, the number of unknowns appearing in the system is $K$ (plus the common boundary value $u$).

This leads us to consider the energy functional
\[\begin{split}
\tilde \cE(\bs{w},\Omega)= \tilde\cK(\bs{w},\Omega)+ \int_{\unl{\Omega}} W(\unl{\bs{w}})\,dx\,,
\end{split}
\]
with
\[\tilde \cK(\bs{w},\Omega) = \frac{1}{2}\littint d\mu(s) \int_{\Omega^+} d(s)\lambda^{1-2s}|\nabla w_s|^2\,dx\,d\lambda\,,\]
where $\Omega\subset\overline{\R^{n+1}_+}$ is open relatively to $\overline{\R^{n+1}_+}$ and is Lipschitz, and $\Omega^+$ and $\unl{\Omega}$ are, respectively, $\Omega\cap\{\lambda>0\}$ and $\Omega\cap\{\lambda=0\}$.
Here, $\bs{w}=\{w_s\}_{s\in{\rm supp}\,\mu}$ denotes a family of bounded functions in $C(\overline{\R^{n+1}_{+})}$ with the property that the traces in $\R^n$ of all the $w_s\in\bs{w}$ coincide.
We then say that such a family $\bs{w}$ has {\em common trace} in $\R^n$ and denote by  $\unl{\bs{w}}$ the function $w_s|_{\{\lambda=0\}}$ (which is the same  for all $s$).

Formally, the Euler-Lagrange equations for  minimizers of $\tilde\cE$ in $\overline{\R^{n+1}_+}$ are \eqref{system_extended}. The following claim relates the ``kinetic parts'' of the energies $\cE$ in $\R^n$ and $\tilde\cE$ in $\overline{\R^{n+1}_+}$.
We outline its proof even if we will not use the claim in the rest of the paper.

\begin{claim}
Let $\varphi$ be such that $\cK(\varphi,\R^n)<\infty$  and, for each $s\in{\rm supp}\,\mu$, let $\tilde \varphi_s$ be the $s$-extension of $\varphi$ to $\R^{n+1}_+$.
Then, the family of $s$-extensions $\bs{\tilde\varphi}=\{\tilde\varphi_s\}_{s\in{\rm supp}\,\mu}$ satisfies
\[\tilde \cK(\bs{\tilde\varphi},\overline{\R^{n+1}_+}) = \cK(\varphi,\R^{n})< \infty\,.\]

Moreover, for every pair of functions $\varphi$, $\psi$ defined in $\R^n$ such that $\cK(\varphi,\R^n)<\infty$, $\cK(\psi,\R^n)<\infty$, and $\varphi\equiv \psi$ outside $B_R$, we have
\[ \tilde \cK(\bs{\tilde \varphi},\overline{\R^{n+1}_+}) -\tilde \cK(\bs{\tilde \psi},\overline{\R^{n+1}_+}) = \cK(\varphi,B_R)- \cK(\psi,B_R)\,,\]
where $\bs{\tilde \varphi}=\{\varphi_s\}$ and  $\bs{\tilde \psi}=\{\psi_s\}$ are the families of $s$-extensions.
\end{claim}
\begin{proof}
We may assume that $\varphi\in C^\infty_c(\R^n)$ by an approximation argument.

Integrating by parts we obtain
\begin{equation}\label{nomnom}
\begin{split}
\int_{\R^{n+1}_+} d(s)\lambda^{1-2s}|\nabla \varphi_s|^2\,dx\,d\lambda &=  -  \int_{\R^{n+1}_+} d(s) {\rm div} \bigl(\lambda^{1-2s} \nabla \varphi_s\bigl) \varphi_s \,dx\,d\lambda \\
& \hspace{30mm} - \lim_{\lambda\searrow 0} \int_{\R^n}  d(s)\lambda^{1-2s} (\partial_\lambda \varphi_s) \varphi_s\,dx
\\
&= 0 + \int_{\R^n}   \varphi (-\Delta)^s \varphi \,dx = 2\cK^s(\varphi, \R^n).
\end{split}
\end{equation}
The first part of the Claim follows integrating \eqref{nomnom} with respect to $d\mu(s)$.

The second part of the Claim is proven similarly.
\end{proof}

From here, by reproducing almost exactly the arguments in \cite{CRS}, next proposition is proven. It extends Lemma 7.2 in \cite{CRS} for nonlocal minimal surfaces  to our situation. Let us point out that the next proposition is not used in the sequel, but we state it since it gives an important structural property of
our extension property. As a consequence of it, we can obtain a close relation between minimizers of $\cE$ in $\R^n$ and of $\tilde \cE$ in $\overline{\R^{n+1}_+}$.

\begin{prop}\label{prop_mzr_baix_implica_mzr_dalt}
Assume that $u: \R^n\rightarrow \R$ is such that $\cK(u,B_1)<\infty$ and let $\bs{\tilde u}=\{\tilde u_s\}_{s\in{\rm supp}\,\mu}$ be the family of $s$-extensions. Let $\varphi\in C^{\infty}_c(B_1)$. Then,
\[\inf_{\Omega, \bs{w}} \littint d\mu(s)\, d(s)  \int_{\Omega^+} \lambda^{1-2s}(|\nabla w_s|^2-|\nabla \tilde u_s|^2)\,dx \,d\lambda= \cK(u+\varphi,B_1)- \cK(u,B_1)\,,\]
where the infimum is taken among all bounded Lipschitz sets $\Omega\subset \overline{\R^{n+1}_+}$, open relatively to $\overline{\R^{n+1}_+}$, and with $\unl{\Omega}\subset B_1$, and among all families $\bs{w}=\{w_s\}$ having common trace $\unl{\bs{w}}= u+\varphi$ in $\R^n$ and such that $w_s-\tilde u_s$ are compactly supported in $\Omega = \Omega^+\cup \unl{\Omega}$.
\end{prop}

Let us now define the notion of minimizers of $\cE$ and $\tilde \cE$.

\begin{defn}\label{def-localmzrs}
We say that $u \in L^{\infty}(\R^n)$ is a minimizer of $\cE$ in $\R^n$ ---given by \eqref{energyfunct}, \eqref{energyfunct1}--- if $|u|\le 1$ in all of $\R^n$ and for every $\Omega\subset\subset \R^n$ we have $\cE(u,\Omega)<\infty$ and
\[\cE(u,\Omega)\le \cE(u+\xi,\Omega)\quad\mbox{for every } \xi\in C^\infty_c(\Omega)\,.\]
\end{defn}

\begin{defn}\label{def-localmzrsadalt}
We say that a family $\bs{v}= \{v_s\}_{s\in{\rm supp}\,\mu}\subset C(\overline{\R^{n+1}_+})$ having common trace $\unl{\bs{v}}$ in $\R^n$ is a minimizer of $\tilde\cE$ in $\overline{\R^{n+1}_+}$ if $|\bs{v}|\le1$ on all of $\overline{\R^{n+1}_+}$ and for every $\Omega\subset \overline{\R^{n+1}_+}$ bounded, Lipschitz, and relatively open, we have $\tilde \cE(\bs{v},\Omega)<\infty$ and
\[\tilde \cE(\bs{v},\Omega)\le \tilde \cE(\bs{v}+\bs{\xi},\Omega)\quad\mbox{for every } \bs{\xi}=\{\xi_s\}\subset C^\infty_c(\Omega) \mbox{ with common trace }\unl{\bs{\xi}}
 \mbox{ in } \R^n.\]
\end{defn}

As a consequence of Proposition \ref{prop_mzr_baix_implica_mzr_dalt}, we have the following link between  minimizers of $\cE$ and of $\tilde\cE$. This is related to Proposition 7.3 in \cite{CRS}.

\begin{prop}\label{prop_mzr_baix_sii_mzr_dalt}
A function $u$ is a minimizer of $\cE$ if, and only if,  the family of $s$-extensions $\bs{\tilde{u}}=\{\tilde u_s\}_{s\in{\rm supp}\,\mu}$ is a minimizer of $\tilde \cE$.
\end{prop}

The following  further relation between $\cE$ and  $\tilde\cE$ is the only one that we will use in the rest of the paper. It applies to functions possibly having infinite energy in all of $\R^n$.
It states that an estimate on the $\cE$ energy in balls for a function $u:\R^n\rightarrow[-1,1]$ (satisfying regularity
estimates) is immediately translated into an estimate of the $\tilde \cE$ energy in cylinders for the family of
$s$-extensions $\{\tilde u_s\}$.
In the remaining part of the paper we denote by $C_R$ the open cylinder in $\overline{\R^{n+1}_+}$ having as bottom $B_R\subset\R^n$ and height $R$ in the $\lambda$ direction:
\begin{equation}\label{cylinder}
C_R=\{(x,\lambda)\,,\ |x|<R\,,\ 0\le\lambda<R\}\,.
\end{equation}

\begin{lem}\label{relationKdown_up_stairs}
Let $u\in C^{2,\gamma}(\R^n)$ with $\|u\|_{C^{2,\gamma}(\R^n)}\le C_0$. Let $\bs{\tilde u}=\{\tilde u_s\}$ be the family of $s$-extensions ---note that $\unl{\bs{\tilde u}}=u$. Then, for $R\ge2$ we have
\[|\tilde \cK(\bs{\tilde u}, C_R) - \cK(u,B_R)|  \le C C_0^2\Phi_{n,s_*}(R)\,,\]
where $C$ depends only on $n$ and $s_*$.
\end{lem}

In the proof of Lemma \ref{relationKdown_up_stairs} we will need the following  elementary  bounds for the extension problems.
\begin{lem}\label{lem:gradbounds}
Assume that $|u|\le C_1$ and $|\nabla u|\le C_2$  in $\R^n$. Then, for $s\in (0,1)$,
the extension of $u$, $\tilde{u}_s$, satisfies
\begin{equation}\label{gradbounds1}
|\tilde{u}_s|\le C_1\quad \mbox{and}\quad \quad |\nabla_x \tilde{u}_s|\le C_2
\end{equation}
in all $\R^{n+1}_{+}$.
Moreover,
\begin{equation}\label{gradbounds2}
|\nabla_x \tilde{u}_s| + |\partial_\lambda \tilde u_s| \le \frac{C C_1}{\lambda} \quad \mbox{for } \lambda>0\,,
\end{equation}
where $C$ depends only on $n$ (and not on $s$).
\end{lem}
\begin{proof}
These bounds are established in \cite[Proposition 4.6.]{C-Si1}. The bounds \eqref{gradbounds1} follow from the maximum principle. The bound \eqref{gradbounds2} follows by interior elliptic estimates, using a scaling argument and observing that for $s\in(0,1)$ the weight $\lambda^{1-2s}$ is uniformly bounded between universal constants in  the domain $\{ 1\le \lambda\le 2\}$.
\end{proof}

We next give the
\begin{proof}[Proof of Lemma \ref{relationKdown_up_stairs}]
By definition
\[2\tilde \cK(\bs{\tilde u}, C_R) = \littint d\mu(s) \int_{B_R}\int_0^R d(s)\lambda^{1-2s}|\nabla \tilde u_s|^2\,dx\,d\lambda\,.\]
Integrating by parts,
\begin{multline}\label{equaciodemoKKtilla}
\int_{B_R}\int_0^R  d(s) \lambda^{1-2s}|\nabla \tilde u_s|^2\,dx\,d\lambda
 =  \int_{\partial B_R} \int_0^R d(s) \lambda^{1-2s} \tilde u_s \frac{\partial \tilde  u_s}{\partial \nu}\,dS\,d\lambda \,+\\ + \int_{B_R} d(s) R^{1-2s}  (\tilde u_s\partial_\lambda \tilde u_s)|_{\lambda=R}\,dx - \int_{B_R} \bigl(\lim_{\lambda\searrow 0} d(s)\lambda^{1-2s} \partial_\lambda \tilde u_s\bigr)\tilde u_s\,dx\,.
\end{multline}

Using the bounds $\eqref{gradbounds1}$ and $\eqref{gradbounds2}$ ---note that the constants $C_1$ and $C_2$ appearing in these bounds are controlled by $C_0$--- we obtain
\[\begin{split}
\biggl|\int_{\partial B_R} \int_0^R d(s) \lambda^{1-2s} \tilde u_s \frac{\partial \tilde u_s}{\partial \nu}\,dS\,d\lambda\biggr| &\le d(s) C_0^2 R^{n-1} \int_0^R \min\{\lambda^{1-2s},\lambda^{-2s}\}\,d\lambda \\&\le  C C_0^2 \Phi_{n,s_*}(R)
\end{split}\]
for every $s\ge s_*$ and for some constant $C$ depending only on $n$ and $s_*$. Here we have used that $d(s)/(1-s) \le C$ as $s\nearrow 1$ (see \cite{C-Si1}).

Similarly, still using $\eqref{gradbounds1}$ and $\eqref{gradbounds2}$,
\[
\biggl|\int_{B_R} d(s) R^{1-2s}  (\tilde u_s\partial_\lambda \tilde u_s)(x,R)\,dx \biggr|\le R^n d(s) R^{1-2s}C_0^2 R^{-1}\le C C_0^2 \,\Phi_{n,s_*}(R)\,.\]

On the other hand, recall that
\[-\lim_{\lambda\to 0^+} d(s)\lambda^{1-2s} \tilde u_s(x,\lambda) \partial_\lambda \tilde  u_s(x,\lambda) = u(x)(-\Delta)^s u(x)\,.
\]

Therefore, integrating \eqref{equaciodemoKKtilla} with respect to $d \mu(s)$ and using the previous bounds, we have proven
\[\biggl|2\tilde \cK(\bs{\tilde u}, C_R) -  \int_{B_R} u Lu\,dx\biggr| \le C C_0^2\Phi_{n,s_*}(R)\,.
\]

Finally, by the formula of integration by parts  \eqref{int_by_parts_formula} we have
\[
\langle u,u\rangle_{\Omega}- \int_{\Omega} L u(x) u(x)\,dx =  \littint d\mu(s) \, c_n(s) \int_{\cC\Omega}dx\int_{\Omega}dy\frac{u(x)-u(y)}{|x-y|^{n+2s}} u(x)\,.
\]

Thus, using Claim \ref{claim:energy} we obtain
\[
\begin{split}
\biggl|\int_{B_R} u Lu\,dx - \langle u,u\rangle_{B_R}\biggr|  &\le  \left| \littint d\mu(s) \, c_n(s) \int_{\cC B_R}dx\int_{B_R}dy\frac{u(x)-u(y)}{|x-y|^{n+2s}} u(x) \right|
\\
&\le  C C_0^2 \left| \littint d\mu(s) \, c_n(s) \int_{\cC B_R}dx\int_{B_R}dy\frac{\min\{1,|x-y|\}}{|x-y|^{n+2s}} \right|
\\
&\le C C_0^2 \Phi_{n,s_*}(R)\,.
\end{split}
\]
Since by definition $\langle u,u\rangle_\Omega=2\cK(u,\Omega)$, the lemma is proved.
\end{proof}

Next we obtain $\tilde\cE$ energy estimates for the family $\bs{\tilde u}$ of $s$-extensions of a layer solution to \eqref{problema1}.

\begin{lem}\label{lem:eg-est-adalt}
Let $u$ be a layer solution in $\R^n$ of $\eqref{problema1}$. Let $\bs{\tilde u}=\{\tilde u_s\}_{s\in{\rm supp}\,\mu}$ be the family of $s$-extensions of $u$ to $\R^{n+1}_+$. Then,
\[\tilde \cE(\bs{\tilde u}, C_R)\le C \Phi_{n,s_*}(R)\,,\]
where $\Phi_{n,s}(R)$ is given by \eqref{eq_phins} and $C$ depends only on $n$, $s_*$, and $W$.
\end{lem}

\begin{proof}
It is a consequence of  Proposition  \ref{thm:eg-est}, combined with Proposition \ref{prop-regularity-mzrs}  and Lemma \ref{relationKdown_up_stairs}.
\end{proof}

\section{Liouville-type theorem and 1-D symmetry}\label{sec7}

In this section we obtain a Liouville theorem within the frame of the extension system \eqref{system_extended}.

\begin{thm}\label{thm_liouville}
Let $\bs{\sigma} = \{\sigma_s\}_{s\in {\rm supp}\,\mu}$  satisfy
\begin{equation}\label{eqthmlouville}
\left\{
\begin{array}{lll}
-\sigma_s\nabla\cdot(\lambda^{1-2s}\varphi_s^2\nabla\sigma_s)\le 0\quad&\text{in }\R^{n+1}_+\,,\ &\text{for each }s\,,\\
\sigma_s(x,0) = \unl{\bs{\sigma}}(x) &\text{on }\R^n\,, &\text{for each }s\,,\\
-\littint d\mu(s)\, {d(s)} \,\unl{\bs{\sigma}} \unl{\bs{\varphi}}^2\, \lim_{\lambda\searrow 0} \lambda^{1-2s} \partial_\lambda \sigma_s   \le 0 \quad&\text{on }\R^n\,,&
\end{array}
\right.
\end{equation}
where $\bs{\varphi}= \{\varphi_s\}_{s\in {\rm supp}\, \mu}$ is a family of positive continuous functions having common trace on $\R^n$.  Assume that $\lambda^{1-2s}\varphi_s^2|\nabla \sigma_s|^2\in L^1_{\text{loc}}(\overline{\R^{n+1}_+})$, for every $s\in {\rm supp}\,\mu$.

Suppose, in addition, that for $R\ge2$,
\begin{equation}\label{hipotesiscreixement}
\littint d\mu(s)\, d(s) \int_{C_{R}} \lambda^{1-2s}(\varphi_s\sigma_s)^2\,dxd\lambda \le C R^2 F(R)\,,
\end{equation}
for some constant $C$ independent of $R$, and some nondecreasing function $F:\R_+\rightarrow\R_+$ such that
\[\sum_{j=1}^{\infty} \frac{1}{F(2^{j+1})}= +\infty\,.\]

Then, $\sigma$ is constant.

\begin{proof}
We adapt the proof of Moschini \cite[Theorem 5.1]{mosch}. Since $\sigma$ satisfies \eqref{eqthmlouville}, we have
\begin{equation}\label{desigthm_lv1}
\nabla\cdot(\sigma_s \lambda^{1-2s}\varphi_s^2\nabla\sigma_s) \ge \lambda^{1-2s}\varphi_s^{2}|\nabla\sigma_s|^2\,,
\end{equation}
for each $s$. On the other hand,
\begin{equation}\label{desigthm_lv2}
\int_{\partial^+C_{R}} \sigma_s \lambda^{1-2s}\varphi_s^2\frac{\partial\sigma_s}{\partial\nu} \,dS \le \biggr(\int_{\partial^+C_{R}}\lambda^{1-2s}\varphi_s^{2}|\nabla\sigma_s|^2\,dS\biggl)^{\frac{1}{2}} \biggr(\int_{\partial^+C_{R}}\lambda^{1-2s}(\varphi_s\sigma_s)^2 \,dS\biggl)^{\frac{1}{2}}\,,
\end{equation}
where $\partial^+C_{R}= \partial C_R \setminus\{\lambda=0\}$, and $\nu$ is the unit outer normal to $\partial^+C_{R}$. Now, set
\[D(R)=\littint d\mu(s)\, d(s) \int_{C_{R}} \lambda^{1-2s} \varphi_s^2 |\nabla\sigma_s|^2\,dx\,d\lambda\,.\]

Let us write $d\tilde\mu(s) = d(s)d\mu(s)$. Using \eqref{desigthm_lv1}, the boundary condition in \eqref{eqthmlouville}, \eqref{desigthm_lv2}, and Schwartz inequality we obtain
\[\begin{split}
D(R)&\le \littint d\tilde\mu(s) \int_{C_{R}} \nabla\cdot(\sigma_s \lambda^{1-2s}\varphi_s^2\nabla\sigma_s)\,dx\,d\lambda \\
&\le \littint d\tilde\mu(s) \int_{\partial^+C_{R}} \sigma_s \lambda^{1-2s}\varphi_s^2\frac{\partial\sigma_s}{\partial\nu} \,dS
\\
&\le\littint d\tilde\mu(s) \biggr(\int_{\partial^+C_{R}}\lambda^{1-2s}\varphi_s^{2}|\nabla\sigma_s|^2\,dS\biggl)^{\frac{1}{2}} \biggr(\int_{\partial^+C_{R}} \lambda^{1-2s}(\varphi_s\sigma_s)^2 \,dS\biggl)^{\frac{1}{2}}
\\
&\le \biggr(\littint d\tilde\mu (s) \int_{\partial^+C_{R}}\lambda^{1-2s}\varphi_s^{2}|\nabla\sigma_s|^2\,dS\biggl)^{\frac{1}{2}} \biggr(\littint d\tilde\mu(s) \int_{\partial^+C_{R}} \lambda^{1-2s}(\varphi_s\sigma_s)^2 \,dS\biggl)^{\frac{1}{2}}
\\
&=D'(R)^{\frac{1}{2}} \biggr(\littint d\tilde\mu(s) \int_{\partial^+C_{R}} \lambda^{1-2s}(\varphi_s\sigma_s)^2 \,dS\biggl)^{\frac{1}{2}}\,.
\end{split}\]
Therefore, if $D(R)>0$,
\begin{equation}\label{eqdemth_lv_3}
\biggr({\littint d\tilde\mu(s)}\int_{\partial^+C_{R}} \lambda^{1-2s}(\varphi_s\sigma_s)^2 \,dS\biggl)^{-1} 
\le \frac{D'(R)}{D(R)^2}.
\end{equation}

Suppose by contradiction that $\bs{\sigma}$ were not constant. Then, for some $R_0>0$, $D(R)>0$ for every $R>R_0$. Integrating \eqref{eqdemth_lv_3} and using Schwartz inequality, we get that, for every $r_2>r_1>R_0$,
\begin{equation}\label{eqdemth_lv_4}
\begin{split}
\frac{1}{D(r_1)}-\frac{1}{D(r_2)}&\ge\int_{r_1}^{r_2}dR\biggr({\littint d\tilde\mu(s)}\int_{\partial^+C_{R}} \lambda^{1-2s}(\varphi_s\sigma_s)^2 \,dS\biggl)^{-1}\\
&\ge(r_2-r_1)^2\biggr({\littint d\tilde\mu(s)}\int_{r_1}^{r_2}dR \int_{\partial^+C_{R}} \lambda^{1-2s}(\varphi_s\sigma_s)^2 \,dS\biggl)^{-1}
\\
&\ge(r_2-r_1)^2\biggr({\littint d\tilde\mu(s)} \int_{C_{r_2}\setminus C_{r_1}} \lambda^{1-2s}(\varphi_s\sigma_s)^2 \,dx\,d\lambda\biggl)^{-1}.
\end{split}\end{equation}

Next, choose $r_2=2^{j+1}$ and $r_1=2^j$ with $j\ge N_0$ such that $2^{N_0}>R_0$. Using \eqref{hipotesiscreixement}, \eqref{eqdemth_lv_4} and summing over $j$, $N_0\le j\le N$, we find
\[\frac{1}{D(2^{N_0})} \ge \frac{1}{4C}\sum_{j=N_0}^{N} \frac{1}{F(2^{j+1})}\,.\]
But, by the hypothesis on $F$, the sum  \[\sum_{j=N_0}^{\infty} \frac{1}{F(2^{j+1})}=+\infty\,,\]
which is a contradiction.
\end{proof}
\end{thm}

We finally prove 1-D symmetry of layer solutions to \eqref{problema1} in dimension two and, with the additional hypothesis $s_*\ge 1/2$, in dimension three.

\begin{proof}[Proof of Theorem \ref{thm:1dsym}]
From $u$ we construct the family of $s$-extensions $\{\tilde u_s\}$.
Given $i<n$, we consider the families $\bs{\sigma^i}= \{( \partial_n \tilde u_s)^{-1} \partial_i \tilde u_s\}$ 
and $\bs{\varphi}=\{\partial_n \tilde u_s\}$.
Observe that both families have common trace, namely, $\unl{\sigma^i}= (\partial_n  u)^{-1}\partial_i  u$ and $\unl \varphi = \partial_n u$ on $\R^n$.
Let us show that these families, for each $i$, satisfy the assumptions of Theorem \ref{thm_liouville}.

Indeed, we have
\[\begin{split}
\nabla\cdot(\lambda^{1-2s}\varphi_s^2\nabla\sigma^i_s)
&= \nabla\cdot\bigl(\lambda^{1-2s}(\partial_n \tilde{u}_s\partial_i \nabla \tilde{u}_s-\partial_i \tilde{u}_s\partial_n\nabla \tilde{u}_s)\bigr)
\\&= \lambda^{1-2s}\bigr(\partial_n\nabla \tilde{u}_s\cdot\partial_i \nabla \tilde{u}_s-\partial_i \nabla \tilde{u}_s\cdot\partial_n\nabla \tilde{u}_s\bigr)
\\&\qquad \qquad + \partial_n \tilde{u}_s\partial_i(\nabla\cdot(\lambda^{1-2s}\nabla \tilde{u}_s))-\partial_i \tilde{u}_s\partial_n(\nabla\cdot(\lambda^{1-2s}\nabla \tilde{u}_s))
\\&=0 \quad\mbox{ in }\R^{n+1}_+,
\end{split}
\]
for each $s\in{\rm supp}\,\mu$.
We now compute the flux on $\R^n=\{\lambda=0\}$. Here we also use the notation $\unl{\lambda^{1-2s}\partial_\lambda v_s}$ for its limit as $\lambda\searrow 0$ (even in cases in which these limits are not common for all $s$). Denoting $d\tilde \mu(s)= d(s)\,d\mu(s)$ we have
\[\begin{split}
\littint  \unl{\sigma^i} \,\unl{\varphi}^2\,\unl{\lambda^{1-2s}\partial_\lambda \sigma^i_s}\, d\tilde\mu(s)
&= \littint \unl{\sigma^i} (\partial_n u)^2 \biggl(\frac{\partial_i(\unl{\lambda^{1-2s}\partial_\lambda \tilde{u}_s})\, \partial_n u
-\partial_n(\unl{\lambda^{1-2s}\partial_\lambda \tilde{u}_s})\,
\partial_i u}
{(\partial_n u)^2}\biggr)\,d\tilde\mu(s)
\\&=\unl{\sigma^i} (\partial_n u\partial_i-\partial_i u\partial_n)\littint \unl{\lambda^{1-2s}\partial_\lambda \tilde{u}_s}\,
d\tilde\mu(s)
\\&=\unl{\sigma^i} (\partial_i u\partial_n-\partial_n u\partial_i) f(u)
\\&=\unl{\sigma^i} (\partial_i u\partial_nu-\partial_nu\partial_iu)f'(u)
\\&\equiv 0\,, \quad \mbox{on }\R^n.
\end{split}
\]

Moreover, by Lemma \ref{lem:eg-est-adalt} and by the assumptions of the theorem, we have that
\[\begin{split}
\littint \,d\tilde \mu(s) \int_{C_{R}} \lambda^{1-2s}(\varphi_s\sigma^i_s)^2\,dx\,d\lambda &=
\littint \,d\tilde \mu(s)\int_{C_{R}} \lambda^{1-2s}(\partial_i u)^2\,dxd\,d\lambda \\
&\le \tilde\cE(\boldsymbol{\tilde u},C_R)\\
&\le C\Phi_{n,s_*} (R)
\end{split}\]
for some constant $C$ independent of $R$.

Next, either if $n=2$ and $s_*\in(0,1)$, or if $n=3$ and $s_*\ge 1/2$, we have
\[ \Phi_{n,s_*} (R) \le C R^2 \log(R).\]

Finally, since the function $F(R)= R^2 \log(R)$ satisfies the assumption of Theorem  \ref{thm_liouville},
it follows that $\bs{\sigma^i}$ is equal to a constant $a^i$, for $i<n$. That is,  $\nabla u =(a^1,1)\partial_2 u$, if $n=2$, or $\nabla u = (a^1,a^2,1)\partial_3 u$, if $n=3$. Equivalently, $u$ has 1-D symmetry.
\end{proof}

Finally, let us point out the following

\begin{rem}\label{Rk1}
Theorem \ref{thm:1dsym} also holds when $\mu(\{1\})>0$, with essentially the same proof. The assumption
$\mu(\{1\})=0$ is done only because it simplifies significantly the notation throughout the paper. Indeed,
when $\mu(\{1\})=0$ we do not need to deal with the two different expressions  \eqref{energyfunct1} and
\eqref{energyfunct1s=1} for the Dirichlet quadratic forms $\mathcal K^s$ ---for $s\in (0,1)$ and $s=1$ respectively---
and only \eqref{energyfunct1} suffices. Up to dealing with lengthier expressions, our proof can be immediately
adapted to the case $\mu(\{1\})>0$. For instance, note that when $\mu(\{1\})>0$ the third equation of the extension
problem \eqref{system_extended} would be replaced by
$$
- \littint d(s)\lim_{\lambda\to 0^+} \lambda^{1-2s} \partial_{\lambda} \tilde{u}_s (x,\lambda) \,d\mu(s) -
\mu(\{1\}) \Delta u = f(u) \qquad\text {on } \{\lambda =0\},
$$
where $f=-W'$. The remaining proofs can be accordingly adapted.
\end{rem}


\end{document}